\documentclass{amsart}
\usepackage[T1]{fontenc}
\usepackage[utf8]{inputenc}
\usepackage{bm}
\usepackage{graphicx, subfigure}
\usepackage{amssymb,latexsym, color}
\usepackage{subfigure}
\usepackage{enumitem}
\usepackage{tikz}
\usetikzlibrary{decorations.markings}

\numberwithin{equation}{section}
\def\C{{\mathbb C}}

\def\Z{{\mathbb Z}}

\def\R{{\mathbb R}}

\newcommand{\Cal}{\mathcal}

\newcommand{\sminus}{\smallsetminus}

\def\res{\mathrm{Res}}
\newcommand{\pp}[1]{\frac{\partial}{\partial #1}}
\newcommand{\tpp}[1]{\tfrac{\partial}{\partial #1 \vphantom{\hat \bX}}}
\newcommand{\bX}{\bm{X}}
\newcommand{\bY}{\bm{Y}}
\newcommand{\bZ}{\bm{Z}}

\usepackage{cleveref}
\theoremstyle{plain}
\newtheorem{lemma}{Lemma}[section]
\newtheorem{proposition}[lemma]{Proposition}

\newtheorem{theorem}[lemma]{Theorem}
\newtheorem{corollary}[lemma]{Corollary}

\theoremstyle{definition}
\newtheorem{definition}[lemma]{Definition}
\newtheorem{example}[lemma]{Example}
\newtheorem{remark}[lemma]{Remark}

\theoremstyle{remark}

\newtheorem{problem}[lemma]{Problem}

\usepackage{xcolor}
\usepackage{xspace}

\newcommand{\fch}{\color{black}\xspace}

\title[Universal unfolding of vector fields in one variable]{On the universal unfolding of vector fields in one variable: A proof of Kostov's theorem}

\author[M. Klime\v{s}]{Martin Klime\v{s}}
\address{Martin Klime\v s\\ ORCiD: 0000-0002-3228-164X.}
\email{klmm@seznam.cz}

\author[C. Rousseau]{Christiane Rousseau}
\address{Christiane Rousseau, D\'epartement de
math\'ematiques et de statistique, Universit\'e de Montr\'eal, C.P. 6128,
Succursale Centre-ville, Montr\'eal (Qc), H3C 3J7, Canada\\
ORCiD: 0000-0002-2021-620X.}
\email{rousseac@dms.umontreal.ca}

\thanks{The second author is supported by NSERC in Canada. }

\subjclass[2010]{37F75, 32M25, 32S65, 34M99} 

\begin{document}

\begin{abstract} In this note we present variants of Kostov's theorem on a versal deformation of a parabolic point of a complex analytic $1$-dimensional vector field. First we provide a self-contained proof of Kostov's theorem, together with a proof that this versal  
deformation is indeed universal. We then generalize to the real analytic and formal cases, where we show universality, and to the $\Cal C^\infty$ case, where we show that only versality is possible. 
\end{abstract}

\keywords{Normal forms of analytic vector fields; unfolding of singularities; versal deformations.}

\maketitle

\section{Introduction} Let us consider a singular point of a germ of analytic vector field $\bX $ on $(\C,0)$. If the singular point is simple, then the germ of vector field is analytically linearizable. If the singular point is multiple,  also called \emph{parabolic}, then the vector field is analytically conjugate to any one of the following normal forms:
\begin{align} 
\bX(x) &=\left( x^{k+1}-\mu x^{2k+1} \right)\pp{x},\label{nf0_1}\\
\bX(x) &= \frac{x^{k+1}}{1+\mu x^k}\pp{x},\label{nf0_2}
\end{align}
where $\mu=\res_{x=0}\bX^{-1}$ is the residue of the dual form.
The first normal form is more frequent in the older works of the Russian school. The second one is easier to manipulate; 
for instance, the rectifying coordinate (time coordinate) $\int\bX^{-1}$ is simple to calculate.

The next natural question is to consider normal forms for unfoldings of germs of analytic vector fields at a singular point. 
When the singular point  is simple the normal form of an unfolding is linear, and hence unique. When the singular point is parabolic, Kostov proved that the following standard deformation of \eqref{nf0_1} is versal \cite{Kostov}:
\begin{equation} 
\bX_1(x;y) =\left( x^{k+1}+y_{k-1} x^{k-1} + \ldots +y_1x+y_0 - (\mu + y_{2k+1}) x^{2k+1}\right)\pp{x}.\label{nf_1}\end{equation}
The proof uses that \eqref{nf_1} is an infinitesimal deformation of \eqref{nf0_1}, and then calls for the machinery of Martinet's Reduction Lemma (see for instance \cite{Arnold0}). 
The philosophy behind this normal form is two-fold. First, a parabolic point of codimension $k$ is the merging of $k+1$ simple singular points, each having its own eigenvalue, which is an analytic invariant. Hence it is natural that a full unfolding would involve $k+1$ parameters. Second, the geometry of an unfolding of a parabolic point is simple, hence the convergence to the normal form. 

Kostov's normal form is very important for many bifurcations problems. For instance, when one studies the unfolding of a parabolic point of a germ of $1$-diffeomorphism of $(\C,0)$ (i.e. a multiple fixed point), then a formal normal form is given by the time one map of a vector field of the form \eqref{nf_1}. The change of coordinate to this normal form diverges and the obstruction to the convergence is the classifying object of the unfoldings (see for instance \cite{MRR}, \cite{Rousseau}, and \cite{Ribon-f, Ribon-a}). The same normal form is used to classify germs of unfoldings of 2-dimensional vector fields in $(\C^2,0)$ with either a saddle-node or a resonant saddle point: indeed the vector fields are orbitally analytically equivalent if and only if the holonomy map of the separatrices (the strong separatrices in the case of a saddle-node) are conjugate (see \cite{RC} and \cite{Rousseau-Teyssier}). 

Very soon, other normal forms equivalent to \eqref{nf_1} appeared in the literature without proof:
\begin{align} 
\bX_2(x;y) &= \left(x^{k+1}+y_{k-1} x^{k-1} + \ldots +y_1x+y_0\right)\!\left(1 - (\mu + y_{2k+1}) x^k\right)\pp{x},\label{nf_2}\\
\bX_3(x;y) &= \frac{x^{k+1}+y_{k-1} x^{k-1} + \ldots +y_1x+y_0}{1+(\mu+y_{2k+1}) x^k}\pp{x},\label{nf_3}
\end{align}
and they are all called \emph{Kostov's theorem}. In practice, most authors use the normal form \eqref{nf_3}, which is much more suitable for computations. 

The paper \cite{Rousseau-Teyssier} indirectly suggests that the normal form \eqref{nf_3} is universal by showing that the normal form \eqref{nf_3} associated to a generic $k$-parameter unfolding of a parabolic point of codimension $k$ is unique up to the action of the group $\Z/k\Z$ of rotations of order $k$. This uniqueness property is extremely important in all classification problems of unfoldings under conjugacy or analytic equivalence: it shows that the parameters of the normal forms are essentially unique and hence analytic invariants of the unfoldings. 
Hence, to show that two unfoldings are analytically equivalent, the first step is to change to the canonical parameters and it then suffices to study the equivalence problem for fixed values of the parameters. 

In this paper, we provide self-contained proofs that the three normal forms \eqref{nf_1}, \eqref{nf_2} and \eqref{nf_3} are  unique up to the action of the group $\Z/k\Z$, and universal.
These self-contained proofs are useful for further generalizations, for instance when the vector field has some symmetry or reversibility property, and also for the formal case, the $\Cal C^\infty$-case, and mixed cases where the variable is analytic and the dependence on the parameters is only real-analytic: this mixed case occurs when one considers bifurcations of antiholomorphic parabolic fixed points (i.e. $f(x) = \bar x \pm \bar x^{k+1} + o(\bar x^{k+1})$).

As a second part of the paper, we briefly address the real analytic, formal,  and smooth cases. 
In the first two cases,  each of the corresponding unfoldings is universal.
In the smooth case, we give an explicit example showing that the unfolding is only versal and cannot be universal, namely the two vector fields 
$\bX(x;\lambda)=(x^2+\lambda^2)\tpp{x}$, and
$\bX'(x;\lambda)=(x^2+(\lambda+\omega(\lambda))^2)\tpp{x}$ are $\Cal C^\infty$-conjugate when  $\omega(\lambda)$ is infinitely flat at $\lambda=0$.
Let us explain one difference with the analytic case.  In the latter case the eigenvalues at the singular points are complex $\Cal C^1$-invariants and 
for a given set of $k+1$ eigenvalues at the singular points, there are only a finite number of solutions for the $y_j$ in any of the normal forms \eqref{nf_1}, \eqref{nf_2} and \eqref{nf_3} with the prescribed eigenvalues. 
In the smooth case, only the eigenvalues at the real singular points are $\Cal C^1$-invariants and we can smoothly glue anything at the complex singular points. The two systems of our counterexample have  purely imaginary singular points. An open question is to know if we have universal unfoldings in the smooth case when all the singular points are real.

The original articles \cite{Kostov, Kostov-real} of Kostov cover a much more general case of deformations of differential forms of real power  $\alpha$. 
However, our goal is not to redo what has been done well by Kostov, but to provide an elementary and self-contained proof in the case of vector fields, that is power $\alpha=-1$, which is why we do not discuss the other cases. Nevertheless, we believe that our proof of the uniqueness in the formal/analytic case, which is missing in Kostov's article, could be well adapted to general $\alpha$.

\section{The analytic theory}

The following definitions are classical: see for instance \cite{Arnold1}.

\medskip

\begin{definition}~
\begin{enumerate}[wide=0pt, leftmargin=\parindent]
	\item  Two germs of analytic (resp. real analytic, formal, $\Cal C^\infty$) parametric families of vector fields $\bX(x;\lambda)$, $\bX'(x';\lambda)$ depending on a same parameter $\lambda$
	are \emph{conjugate} if there exists an  analytic (resp. real analytic, formal, $\Cal C^\infty$) invertible change of coordinate
	\begin{equation}\label{eq:transformation}
	 x'=\phi(x;\lambda), 
	\end{equation} 
	changing one family to the other. We write $\bX=\phi^*\bX'$  as a \emph{pullback} of $\bX'$. 

	\item Let $\lambda\mapsto \lambda'=\psi(\lambda)$ be a germ of analytic (resp. real analytic, formal, $\Cal C^\infty$) map (not necessarily invertible), then $\bX(x,\lambda)=\bX'(x,\psi(\lambda))$ is a family \emph{induced} from $\bX'$.	
	
	\item A parametric family of vector fields $\bX(x;\lambda)$ is a \emph{deformation} of $\bX(x;0)$. Two deformations $\bX(x;\lambda)$, $\bX'(x,\lambda)$ of the same initial vector field $\bX(x;0)=\bX'(x;0)$ with the same parameter $\lambda$ are \emph{equivalent} (as deformations) if the two families are conjugate by means of an invertible transformation \eqref{eq:transformation} with $\phi(x;0)\equiv x$.
	
	\item A deformation $\bX'(x,\lambda')$ of $\bX'(x,0)$ is \emph{versal} if any other deformation $\bX(x,\lambda)$ of $\bX'(x,0)= \bX(x,0)$ is equivalent to one induced from it. 
	It is \emph{universal} if the inducing map $\lambda'=\psi(\lambda)$ is unique.  
\end{enumerate}	
\end{definition}

In this section we provide a self-contained proof of the following theorem:

\begin{theorem}\label{thm1} 
In the analytic case, for $k\geq 1$, the deformation \eqref{nf_3} of \eqref{nf0_2} is universal.
\end{theorem}

\begin{corollary}\label{cor:Kostov} 
In the analytic case, for $k\geq 1$, the deformations  \eqref{nf_1} and \eqref{nf_2} of \eqref{nf0_1} are universal.
\end{corollary}

As explained in the introduction, the proof of the versality is due to Kostov (for \eqref{nf_1} see \cite{Kostov},
while for \eqref{nf_3} it has been often stated in literature without explicit proof, see e.g. \cite[p.116]{Arnold}), 
and the uniqueness comes from \cite{Rousseau-Teyssier}. 
Theorem~\ref{thm1} can be rephrased in more precise terms as the following theorem of which it is a direct consequence.

\begin{theorem}\label{theorem:Kostov}~
\begin{enumerate}[label=(\roman*), leftmargin=\parindent, itemindent=0em, itemsep=4pt, topsep=4pt, partopsep=0pt, parsep=0pt]
\item \textnormal{(Kostov \cite{Kostov})} Any  analytic germ of a family of vector fields $\tilde \bX(x,\lambda)$ depending on a multi-parameter $\lambda$ unfolding $\tilde \bX(x;0)=x^{k+1}\frac{1}{\omega(x)}\pp{x}$, $\omega(0)\neq 0$, $k\geq 0$, is analytically conjugate
to a family of the form 
	\begin{align} 	\label{Kostov_form0}
	\bX(x;\lambda)&=c(\lambda)x\pp{x}, & k&=0,\\
	\bX(x;\lambda)&=\frac{x^{k+1}+y_{k-1}(\lambda)x^{k-1}+\ldots+y_0(\lambda)}{1+\mu(\lambda)x^k}\pp{x},& k&\geq 1,
	\label{Kostov_form}
	\end{align}
with $y_0(0)=\ldots =y_{k-1}(0)=0$, where
\[
\mu(\lambda)=-\res_{x=\infty}\bX(x;\lambda)^{-1}
\]	
is the sum of the residues of $\tilde\bX(x;\lambda)^{-1}$ over its local polar locus around the origin.

\item  \textnormal{(Rousseau, Teyssier \cite[Theorem 3.5]{Rousseau-Teyssier}, \cite[Theorem 7.2]{Klimes-Rousseau})}
The normal form \eqref{Kostov_form0} for $k=0$ and \eqref{Kostov_form} for $k=1$ are unique, while the normal form
\eqref{Kostov_form} for $k > 1$ is unique up to the action of $x\mapsto e^{2\pi i\frac{l}{k}}x$, $l\in\Z_k$,
$$y_{j}(\lambda)\mapsto e^{-2\pi i\frac{(j-1)l}{k}}y_{j}(\lambda),\quad j=0,\ldots,k-1.$$
More precisely, if $x\mapsto\phi(x,\lambda)$ is a transformation  between two vector fields \eqref{Kostov_form}, fixing $\lambda$,
then 
\[
\phi(x,\lambda)=\begin{cases} e^{t(\lambda)}x, & \text{if }\ k=0,\\
 e^{2\pi i\frac{l}{k}} \exp(t(\lambda)\bX)(x;\lambda),  & \text{if }\ k\geq 1,\end{cases}
\]
for some $l\in\Z_k$ and some analytic germ $t(\lambda)$.
\end{enumerate}	
\end{theorem}

The first step in proving Theorem~\ref{theorem:Kostov} is the following ``prenormal form'', which can be also found for example in \cite[Proposition~5.13]{Ribon-f}.

\begin{proposition}[Prenormal form] \label{prop:prenormal}
Any germ of a family of vector fields $\tilde \bX(x,\lambda)$ depending on a multi-parameter $\lambda$ unfolding $\tilde \bX(x;0)=(cx^{k+1}+\ldots)\pp{x}$, $k\geq 1$, is analytically conjugate to a family of the form 
\begin{equation} 
\bX(x;\lambda)=\frac{x^{k+1}+y_{k-1}(\lambda)x^{k-1}+\ldots+y_0(\lambda)}{1+u_0(\lambda)+\ldots+u_{k-1}(\lambda)x^{k-1}+\mu(\lambda)x^k}\pp{x},
\label{pol_form}
\end{equation}
where $y_j(0)=0=u_j(0)$, $j=0,\ldots,k-1$, and
\[
\mu(\lambda)=-\res_{x=\infty}\bX(x;\lambda)^{-1}.
\]
\end{proposition}

\begin{proof} 
First, let us transform $\tilde \bX(x;0)=x^{k+1}\frac{1}{\omega(x)}\pp{x}$ to a form $\bX(x;0)=\frac{x^{k+1}}{1+\mu(0)x^k}\pp{x}$. 
Up to a linear change $x\mapsto ax$, $a\in\C\sminus\{0\}$, we can assume $\omega(0)=1$. 
Write $\omega(x)=1+\omega_1x+\ldots+\omega_kx^k+x^{k+1}r(x)$, 
let $\mu(0):=\omega_k$, and let 
$\alpha(x):=\int\tilde \bX(x;0)^{-1}-\bX(x;0)^{-1}=-\tfrac{\omega_1}{k-1}x^{1-k}-\ldots-\frac{\omega_{k-1}}{1}x^{-1}+\int_0^xr(x)dx$. 
Then $\alpha$ is a  meromorphic germ with pole of order at most $k-1$ at the origin, and the desired transformation is provided by Lemma~\ref{lemma} below.
	
By Weierstrass preparation and division theorem, any family	$\tilde \bX(x,\lambda)$ can be written in the form
$\tilde \bX(x,\lambda)=\frac{P(x;\lambda)}{Q(x;\lambda)+P(x,\lambda)R(x,\lambda)}\pp{x}$, for some Weierstrass polynomials $P(x;\lambda)=x^{k+1}+y_{k-1}(\lambda)x^{k-1}+\ldots+y_0(\lambda)$,
$Q(x;\lambda)=1+u_0(\lambda)+\ldots+u_{k-1}(\lambda)x^{k-1}+\mu(\lambda)x^k$, and some analytic germ $R(x;\lambda)$. Let $\alpha(x,\lambda)=\int R(x;\lambda)dx$, then
\[\tilde \bX(x,\lambda)=\frac{\bX(x,\lambda)}{1+\bX(x,\lambda).\alpha(x,\lambda)}\] 
for $\bX(x;\lambda)=\frac{P(x;\lambda)}{Q(x;\lambda)}\pp{x}$ of the form \eqref{pol_form}.
The result follows from Lemma~\ref{lemma}.
\end{proof}

The following lemma is classical (see for example  \cite[Proposition 2.2]{Teyssier} to which it is essentially equivalent).

\begin{lemma}\label{lemma}
	Let $\bX_0,\ \bX_1$ be two germs of analytic families of vector fields vanishing at the origin, and assume there exists an analytic germ $\alpha(x,\lambda)$ such that $\bX_1=\frac{\bX_0}{1+\bX_0.\alpha}$.
	Then the flow map of the vector field $\bY(x,t;\lambda)=\pp{t}-\frac{\alpha \bX_0}{1+t\bX_0.\alpha}$
	\[	 \phi_1(x,\lambda)=x\circ\exp(\bY)\Big|_{t=0}, \]
	conjugates $\bX_1$ with $\bX_0=\phi_1^*\bX_1$. 
	
The statement is also true if $\alpha$ is meromorphic such that $\bX_0.\alpha$ and $\alpha\bX_0$ are analytic and vanish for $(x,\lambda)=0$ (so that the flow of $\bY$ is defined for all $t\in[0,1]$).
\end{lemma}

\begin{proof}
	On the one hand, if $\bX_t:=\frac{\bX_0}{1+t\bX_0.\alpha}$ and $\bY=\pp{t}-\frac{\alpha \bX_0}{1+t\bX_0.\alpha}$
	are vector fields in $x,t,\lambda$, then $[\bY,\bX_t]=0$, which means that the flow $\exp(s\bY):(x,t)\mapsto(\Phi_s(x,t),t+s)$ of $\bY$ preserves $\bX_t=\Phi_s^*\bX_{t+s}$. In particular $\phi_s(x):=\Phi_s(x,0)$ is such that $\phi_s^*\bX_s=\bX_0$.
 \end{proof}

\begin{proposition}\label{prop:normalisation}
Consider two families of vector fields $\bX_0$ and $\bX_1$ of the form 
\begin{equation} 
\bX_t=\frac{x^{k+1}+y_{k-1}x^{k-1}+\ldots+y_0}{1+t(u_0+\ldots+u_{k-1}x^{k-1})+\mu x^k}\pp{x},
\label{pol_form2}
\end{equation}
depending on parameters $(y,u,\mu)$.
Then there exists an analytic transformation $(x,y)\mapsto\big(\phi(x,y;u,\mu),\psi(y;u,\mu)\big)$ tangent at identity at $(x,y)=0$ that  conjugates
$\bX_1$ to $\bX_0$.
\end{proposition}

\begin{proof}
Let $\bX_t=\frac{P(x,y)}{Q(x,t;u,\mu)}\pp{x}$ be as above \eqref{pol_form2}.	
We want to construct a family of transformations depending analytically on $t\in[0,1]$ between $\bX_0$ and $\bX_t$, defined by a flow of a vector field $\bY$ of the form
\[
\bY=\pp{t}+\sum_{j=0}^{k-1}\xi_j(t,y;u,\mu)\pp{y_j}+\frac{H(x,t,y;u,\mu)}{Q(x,t;u,\mu)}\pp{x},
\]
for some $\xi_j$ and $H$, such that $[\bY,\bX_t]=0$, that is
\[
-\frac{UP}{Q^2}+\frac{\Xi}{Q}+\frac{H}{Q}\pp{x}\left(\frac{P}{Q}\right)-\frac{P}{Q}\pp{x}\left(\frac{H}{Q}\right)=0,
\]
where $U(x;u)=u_0+\ldots+u_{k-1}x^{k-1}$ and $\Xi(x;\xi)=\xi_0+\ldots+\xi_{k-1}x^{k-1}$, which is equivalent to
\begin{equation} 
H\tpp{x}P-P\tpp{x}H+Q\Xi=UP.
\label{homological_eq}
\end{equation}
We see that we can choose $H$ as a polynomial in $x$:
\[
H=h_0(t,y;u,\mu)+\ldots+h_k(t,y;u,\mu)x^k.
\]
Write $UP=b_0(y;u)+\ldots+b_{2k}(y;u)x^{2k}$,
then the equation \eqref{homological_eq} takes the form of a non-homogeneous linear system for the coefficients $(\xi,h)$:
\[
A(t,y;u,\mu)\begin{pmatrix}\xi\\ h\end{pmatrix}=b(y;u).
\]
For $y=u=0$ the equation \eqref{homological_eq} is
\[
(k+1)x^kH-x^{k+1}\tpp{x}H+(1+\mu x^k)\Xi=0,
\]
hence
\[
A(t,0;0,\mu)=
\begin{pmatrix} 1&& &0&& &0\\ &\ddots& &&\ddots& &\\ &&1 &&&0 &0\\ \mu&& &k+1&& &0\\ &\ddots& &&\ddots& &\\ &&\mu &&&2 &0\\ 0&\ldots&0 &0&\ldots&0 &1 \end{pmatrix}.
\]
This means that $A(t,y;u,\mu)$ is invertible for $(t,\mu)$ from any compact in $\C\times\C$ if $|y|,\ |u|$ are small enough.
Since $b(0;0)=0$, the constructed vector field $\bY(x,t,y;u,\mu)$ is such that $\bY(0,t,0;0,\mu)=\pp{t}$ and its flow is well-defined for all $|t|\leq 1$ as long as  $|y|,\ |u|$ are small enough.
\end{proof}

\medskip

\begin{proof}[Proof of Theorem~\ref{theorem:Kostov}]~
\begin{enumerate}[label=(\roman*), leftmargin=\parindent, itemindent=0em, itemsep=4pt, topsep=4pt, partopsep=0pt, parsep=0pt]
	\item The existence of an analytic normalizing transformation to \eqref{Kostov_form} when $k\geq1$ follows directly from Propositions~\ref{prop:prenormal} and~\ref{prop:normalisation}. For $k=0$ it follows from Lemma~\ref{lemma}.
	
	\item Let us prove the uniqueness.
	For $k=0$ it is obvious. 
	For $k>0$, let $\phi(x;\lambda)$ be a transformation between 
	$\bX=\frac{P(x;\lambda)}{1+\mu(\lambda) x^k}\pp{x}$ and $\bX'=\frac{P'(x;\lambda)}{1+\mu'(\lambda) x^k}\pp{x}$, preserving the parameter $\lambda$, and such that $\phi^*\bX=\bX'$.
    By the invariance of the residue, $\mu(\lambda)=\mu'(\lambda)$.
    
    Let $\phi(x;0)=cx+\ldots$ for some $c\neq 0$; necessarily $c=e^{2\pi i\frac{l}{k}}$ for some $l\in\Z_k$. Up to precomposition with a map $x\mapsto cx$,
    we can assume that $c=1$ and that $\phi(x;0)=x+\ldots$ is tangent to identity.
    Let 
    $$G(x,t, \lambda)=\exp(-t \bX)\circ \phi(x,\lambda),$$ 
    and 
    $$K(t,\lambda)=\frac{\partial^{k+1}G}{\partial x^{k+1}}\big|_{x=0}.$$ 
    The map $K$ is analytic and $\frac{\partial K}{\partial t}(t,0)= -(k+1)!\neq0$. 
    For $\lambda=0$, there exists $t_0$ such that $K(t_0,0)=0$ (in fact $t_0=\frac{1}{(k+1)!}\frac{\partial^{k+1}\phi}{\partial x^{k+1}}(0;0)$ since  $\exp(t\bX(x;0))=x+tx^{k+1}+\ldots$). 
    By the implicit function theorem, there exists a unique function $t(\lambda)$ such that $K(t(\lambda), \lambda)\equiv 0$. Then considering the new transformation $\psi=\exp(-t(\lambda)X)\circ\phi$, it suffices to proves that $\psi\equiv id$. This is done by the infinite descent. 
    
    Let  $\psi(x,\lambda)=x+f(x;\lambda)$, where $\frac{\partial ^{k+1} f}{\partial x^{k+1}}\equiv 0$. 
      Denote $\Cal I_\lambda$ the ideal of analytic functions of $(x;\lambda)$ that vanish when $\lambda=0$. To show that $\psi(x,0)\equiv x$ it suffices to show that  $f\in\Cal I_\lambda^n$ for all $n$. 
     For $\lambda=0$ both vector fields $\bX(x;0)$ and $\bX'(x;0)$ are equal to $\frac{x^{k+1}}{1+\mu(0)x^k}\pp{x}$, 
     and it is easy to verify that  $\psi(x,0)\equiv x$ (for instance using power series), which gives us the induction hypothesis $f\in \Cal I_\lambda$. 
     
     Suppose now that $f\in \Cal I_\lambda^n$.
     Developing the right side of the transformation equation 
     \begin{equation*}
     \tfrac{P}{1+\mu x^k}\tpp{x}\psi=\tfrac{P'\circ\psi}{1+\mu\psi^k},
     \end{equation*} 
     we have
    \[\tfrac{P}{1+\mu x^k}\cdot\tpp{x}(x+f)=\tfrac{P'}{1+\mu x^k}+ f\cdot\tpp{x}\tfrac{P'}{1+\mu x^k}\mod \Cal I_\lambda^{n+1},\]
    from which
    \[P-P'=-x^{k+1}\tpp{x}f +f\cdot\big((k+1)x^k-\tfrac{k\mu(0)x^{2k}}{1+\mu(0)x^k}\big) \mod \Cal I_\lambda^{n+1}.\]
    The left side being a polynomial of order $\leq k-1$ in $x$, this means that both sides vanish modulo $\Cal I_\lambda^{n+1}$. 
   Therefore on the left side $P=P'\mod \Cal I_\lambda^{n+1}$, while the right side can be rewritten as 
    \[-(1+\mu(0)x^k)x\tpp{x}f +f\cdot ((k+1)+\mu(0)x^{k})\equiv 0 \mod \Cal I_\lambda^{n+1},\]
    Putting $f=\sum_{j=0}^\infty f_j x^j$, yields
    \[\sum_{j=0}^{k-1} (k+1-j)f_j x^{j} +\sum^\infty_{j=k} (k+1-j)(f_j+\mu(0)f_{j-k})x^{j} \equiv 0 \mod \Cal I_\lambda^{n+1},\]
    from which we get that all $f_j\in\Cal I_\lambda^{n+1}$ since $f_{k+1}\equiv0$. \qedhere
\end{enumerate}	
\end{proof}

\begin{proof}[Proof of Corollary~\ref{cor:Kostov}]
Consider the two families \eqref{nf_1} and \eqref{nf_3}
\begin{align*}
\bX_1(x;y) &= x^{k+1}+y_{k-1} x^{k-1} + \ldots +y_1x+y_0 - (\mu + y_{2k+1}) x^{2k+1}\pp{x},\\
\bX_3(x';y') &= \frac{x'^{k+1}+y'_{k-1} x'^{k-1} + \ldots +y'_1x'+y'_0}{1+(\mu+y'_{2k+1}) x'^k}\pp{x'}.
\end{align*}
By Theorem~\ref{theorem:Kostov} we know that there exists a map
\[x'=\phi(x,y),\quad y_j'=\psi_j(y),\ j=0,\ldots,k-1,\ 2k+1, \]
such that $\bX_1(x,y)=\phi^*\bX_3(x,\psi(y))$, that is, such that
\[\begin{split}
&\frac{\phi^{k+1}+\psi_{k-1} \phi^{k-1} + \ldots +\psi_1\phi+\psi_0}{1+(\mu+\psi_{2k+1}) \phi^k}= \\
&\qquad\qquad=\left(x^{k+1}+y_{k-1}x^{k-1}  + \ldots  +y_1 x+y_0 - (\mu  + y_{2k+1})x^{2k+1}\right)\pp{x}\phi.
\end{split}\]
We want to show that $\psi$ is invertible.

For $y=0$ we have 
\[\frac{\phi(x,0)^{k+1}}{1+\mu \phi(x,0)^k}=\left(x^{k+1}-\mu x^{2k+1}\right)\pp{x}\phi(x,0),\]
and (up to a pre-composition with a flow map of $\bX_1$ killing the term in $x^{k+1}$) we can assume that
\[\phi(x,0)=x+\mu^2\tfrac{1}{k}x^{2k+1}+O(x^{2k+2}).\] 
Write $\phi(x,y)=\phi(x,0)+f(x,y)$, with $f(x,y)=\sum_{l=0}^\infty f_l(y)x^l$, and denote $\Cal I_y$ the ideal of functions that vanish when $y=0$.
Then calculating modulo $\Cal I_y^2$:
\[\begin{split}
&\frac{\phi(x,0)^{k+1}+\psi_{k-1} \phi(x,0)^{k-1} + \ldots +\psi_1\phi(x,0)+\psi_0}{1+(\mu+\psi_{2k+1}) \phi(x,0)^k}\\
&\qquad +
\frac{(k+1)\phi(x,0)^k+\mu\,\phi(x,0)\fch^{2k}}{(1+\mu\,\phi(x,0)^k)^2}f(x,y)=\\
&=\left(x^{k+1}+y_{k-1}x^{k-1} + \ldots  +y_1 x+y_0 - (\mu  + y_{2k+1})x^{2k+1}\right)\pp{x}\phi(x,0)\\
&\qquad +\left(x^{k+1}-\mu x^{2k+1}\right)\pp{x}f(x,y) \mod\Cal I_y^2.
\end{split}\]
Comparing the coefficients of $x^j$, $j=0,\ldots,k-1$, on both sides we have
\[\psi_j=y_j\mod\Cal I_y^2,\quad j=0,\ldots,k-1,\]
for $j=k+1$
\[-\psi_1\mu +(k+1)f_1(y)=f_1(y) \mod\Cal I_y^2,\]
and for $j=2k+1$
\[\begin{split}
&-\mu-\psi_{2k+1}+\mu^2\tfrac{1}{k}\psi_1-(2k+1)\mu f_1+(k+1)f_{k+1}=\\
&\qquad =-\mu - y_{2k+1}+\mu^2\tfrac{2k+1}{k}y_1+(k+1)f_{k+1}-\mu f_1\mod\Cal I_y^2,
\end{split}\]
from which 
\[\psi_{2k+1}=y_{2k+1}-4\mu^2y_1 \mod\Cal I_y^2.\]
This means that the transformation $(x,y)\mapsto(\phi(x,y),\psi(y))$ is invertible for small $x,y$. 

Similarly for the families \eqref{nf_2} and \eqref{nf_3}
\end{proof}

\section{Real analytic, formal and smooth theory}

\begin{theorem}[Real analytic theory]\label{theorem:Kostov-real}~
The statement of Theorem~\ref{theorem:Kostov} is also true in the real analytic setting, 
with the exception that \eqref{Kostov_form} needs to be replaced by
\begin{equation} 	\label{Kostov_form-real}
\bX_{\text{real}}(x;\lambda)=\frac{x^{k+1}+y_{k-1}(\lambda)x^{k-1}+\ldots+y_0(\lambda)}{(\pm 1)^{k+1}+\mu(\lambda)x^k}\pp{x}.
\end{equation}
Consequently, the real analytic parametric family
\begin{equation} 
\bX_{3,\text{real}}(x;y) = \frac{x^{k+1}+y_{k-1} x^{k-1} + \ldots +y_0}{(\pm 1)^{k+1}+(\mu+y_{2k+1}) x^k}\pp{x},\label{nf_4}
\end{equation}
is a universal real analytic deformation for $\bX_{3,\text{real}}(x;0)$.
\end{theorem}

\begin{proof}
Assuming the initial vector field $\tilde\bX$ is real analytic then so are all the transformations of Propositions~\ref{prop:prenormal} and \ref{prop:normalisation} with the only exception: the leading coefficient of $\tilde\bX(x;0)=\big(cx^{k+1}+\ldots\big)\pp{x}$ can be brought to either $c=\pm 1$ if $k$ is even, and $c=1$ if $k$ is odd. 	
\end{proof}

\begin{corollary} 
If we consider deformations which are symmetric  (resp. antisymmetric (also called reversible)) with respect to the real axis, then their associated universal deformations
\begin{align*} 
\bX_1'(x;y) &=c\left(x^{k+1}+y_{k-1} x^{k-1} + \ldots +y_1x+y_0 - c\,(\mu+y_{2k+1}) x^{2k+1}\right)\pp{x},\\
\bX_2'(x;y) &= c\left(x^{k+1}+y_{k-1} x^{k-1} + \ldots +y_1x+y_0\right)\!\left(1 - c\,(\mu+y_{2k+1}) x^k\right)\pp{x},\\
\bX_3'(x;y) &= \frac{x^{k+1}+y_{k-1} x^{k-1} + \ldots +y_1x+y_0}{c+(\mu+y_{2k+1}) x^k}\pp{x},
\end{align*}
$y_0,\ldots y_{k-1},cy_{2k+1}\in \R$, $c\mu\in\R$, with $c=(\pm 1)^{k+1}$ (resp. $c=i\,(\pm 1)^{k+1}$), have the same property, and the conjugacy commutes with the symmetry.
\end{corollary}


\begin{theorem}[Formal theory]\label{theorem:Kostov-formal}~
	\begin{enumerate}[leftmargin=\parindent, itemindent=0em, itemsep=4pt, topsep=4pt, partopsep=0pt, parsep=0pt]
		\item 	The statement of Theorem~\ref{theorem:Kostov} and therefore of Theorem~\ref{thm1} is also true in the formal setting, of formal parametric germs of vector fields and formal transformations \eqref{eq:transformation}, where by formal we mean a formal power series in $(x,\lambda)$. 
		In the formal real case (i.e. the series have real coefficients), then the normal form is given by \eqref{Kostov_form-real}.
		
		\item \textnormal{(Ribon \cite[Proposition 6.1]{Ribon-f})}
		Two \emph{analytic} germs of vector fields $\bX,\bX'$ are formally conjugate if and only if they are analytically conjugate.	
		
		Moreover, denoting $\hat{\Cal I}_\lambda$ the ideal of formal series that vanish when $\lambda=0$, if $\hat\phi(x,\lambda)$ is a formal conjugating transformation, then for any $n>0$ there exists an analytic conjugacy $\phi_n(x,\lambda)$, $\phi_n^*\bX'=\bX$, such that $\phi_n=\hat\phi\mod\hat{\Cal I}_\lambda^n$.
	\end{enumerate}
\end{theorem}

The second statement is an analogue of the Artin approximation theorem.

\begin{proof}
	\begin{enumerate}[leftmargin=\parindent, itemindent=0em, itemsep=4pt, topsep=4pt, partopsep=0pt, parsep=0pt]
		\item The proof follows exactly the same lines. The key fact is that a formal flow map of a formal vector field 
		\[
		\hat\bY=\pp{t}+\sum_{j=0}^{k-1}\hat\xi_j(t,y;u,\mu)\pp{y_j}+\hat F(x,t,y;u,\mu)\pp{x},
		\] 
		which is analytic in $t$ and the parameters $(u,\mu)$ is well defined: see Lemma~\ref{lemma:formalflow} below.
		
		\item This is a consequence of the uniqueness of the normal form \eqref{Kostov_form}: each analytic germ of parametric vector field is analytically conjugate to a normal form \eqref{Kostov_form}, and two such normal forms are formally conjugate if and only if they are conjugate by a rotation $x\mapsto e^{2\pi i\frac{l}{k}}x$, $l\in\Z_k$, which is analytic. 	
		 Moreover, the formal conjugacy is a composition of the rotation and of a formal time $\hat t(\lambda)$-flow map of the vector field. Replacing $\hat t(\lambda)$ with an analytic $t_n(\lambda)=\hat t(\lambda)\mod\hat{\Cal I}_\lambda^n$ does the trick. \qedhere
	\end{enumerate}
\end{proof}

\begin{lemma}\label{lemma:formalflow}
	Let  $\hat\bY=\tpp{t}+\hat\bZ(z,t)$, where $\hat\bZ(z,t)$ be a formal vector field in $z\in\C^p$ with coefficients entire in $t\in\C$, that vanishes at $z=0$: $\hat\bZ(0,t)=0$.
	Then $\hat\bY$ has a well-defined formal flow $z\circ\hat\exp(s\hat\bY)=\sum_{n=0}^{+\infty}\frac{s^n}{n!}\hat\bY^n.z$ for all $(s,t)\in\C^2$.
\end{lemma}
\begin{proof}
	For any $n\in\Z_{\geq 0}$, the $n$-jet with respect to the variable $z$ of $\hat\bY$ is an entire vector field $j_z^n\hat\bY(z,t)$ in $\C^p\times\C$ with well defined flow $z\circ\exp(s\, j_z^n\hat\bY)$ fixing the origin in $z$.
	For any $m\leq n$, the $m$-jet of this flow agrees with the $m$-jet of the one for $m$: 
	$j_z^m\big(z\circ\exp(s\,j_z^n\hat\bY)\big)=j_z^m\big(z\circ\exp(s\,j_z^m\hat\bY)\big)$,
	meaning that they converge in the Krull topology as $n\to+\infty$ to a well-defined formal flow map $z\circ\hat\exp(s\hat\bY)$.	
	See also \cite[Theorem~3.9]{IlYa}.
\end{proof}

\begin{theorem}[$\Cal C^\infty$-smooth theory, Kostov \cite{Kostov-real}]\label{theorem:Kostov-real}~\\
 In real $\Cal C^\infty$-smooth setting, the deformation $\bX_{3,\text{real}}(x;y)$ \eqref{nf_4} is a versal deformation of the normal form vector field $\bX_{3,\text{real}}(x;0)$.
\end{theorem}

\begin{proof} The only purely analytic tools used in the proof of the existence of a normalizing transformation were the Weierstrass preparation and division theorems (used in the proof of Proposition~\ref{prop:normalisation}), which have their counterpart in the $\Cal C^\infty$-setting in the Malgrange preparation and  division theorems \cite{Malgrange, GG}. 
\end{proof}

The deformation \eqref{nf_4} is not universal in the  $\Cal C^\infty$-setting in general. 
The issue is the non-uniqueness in the Malgrange division and the lack of control over the potential non-real singularities in the family.

\begin{example}
The deformations
$\bX(x;\lambda)=(x^2+\lambda^2)\tpp{x}$, and
$\bX'(x;\lambda)=\big(x^2+(\lambda+\omega(\lambda))^2\big)\tpp{x}$, where $\omega(\lambda)$ is infinitely flat at $\lambda=0$ 
(i.e. $\big(\tpp{\lambda}\big)^n\omega\big|_{\lambda=0}=0$, for all $n\in\Z_{\geq 0}$),
are $\Cal C^\infty$-equivalent by means of a conjugacy $\phi(x;\lambda)$ with $\Delta(x;\lambda):=\phi(x;\lambda)-x$ infinitely flat along $\lambda=0$.

Indeed we first change $x\mapsto \frac{\lambda+\omega(\lambda)}{\lambda}x =C(\lambda)x$ into $\bX'$, thus transforming it into 
$\bX''(x;\lambda)=\frac1{C(\lambda)}\bX(x;\lambda)$. Note that $C(\lambda) = 1+\frac{\omega(\lambda)}{\lambda}= 1+\mu(\lambda)$, with $\mu(\lambda)$ infinitely flat. 
We then apply Lemma~\ref{lemma}. We look for a germ $\alpha(x;\lambda)$ such that 
$1+\bX.\alpha= C(\lambda)$, which is equivalent to $(x^2+\lambda^2)\tpp{x}\alpha= \mu(\lambda)$. This equation has the odd solution
$$\begin{cases}
\alpha(x;\lambda) = \frac{\mu(\lambda)}{\lambda} \arctan \frac{x}{\lambda},&\lambda\neq0,\\
0, &\lambda=0.\end{cases}$$
The function $\alpha$ is obviously $\Cal C^\infty$ since $\arctan\frac{x}{\lambda}$ is bounded, and each derivative of $\arctan\frac{x}{\lambda}$  grows no faster than $(x^2+\lambda^2)^{-n}<\lambda^{-2n}$ for some $n$ depending on the derivative.
The flow of the vector field 
\[\bY=\pp{t}-\frac{\alpha\bX}{1+t\bX.\alpha}=
\pp{t}-\frac{\alpha(x,\lambda)(x^2+\lambda^2)}{1+t\mu(\lambda)}\pp{x}\]
is well defined and $\Cal C^\infty$-smooth for $t\in[0,1]$ as long as $|\mu(\lambda)|<1$. 
\end{example}

\begin{remark} The deformations
$\bX(x;\lambda)=(x^2-\lambda^2)\tpp{x}$, and
$\bX'(x;\lambda)=\big(x^2-(\lambda+\omega(\lambda))^2\big)\tpp{x}$, where $\omega(\lambda)$ is infinitely flat at $\lambda=0$, are not $\Cal C^\infty$-conjugate. Indeed, the eigenvalues at the singular points are $\Cal C^1$ invariants. 
\end{remark}

\begin{problem}
Can we expect uniqueness   of the induced coefficients in \eqref{Kostov_form-real} in the special case when the deformation is such that it has $k+1$ merging real singularities when counted with multiplicity?
\end{problem}

\bibliographystyle{alpha}

\end{document}